\documentclass[12pt]{ip-journal}

\usepackage[marginratio=1:1]{geometry}
\usepackage[english]{babel}
\usepackage{amsmath}
\usepackage{graphicx}
\usepackage{dsfont}
\usepackage{epstopdf} 
\usepackage{array}
\usepackage{latexsym}
\usepackage[font=small]{caption}
\usepackage[hidelinks]{hyperref}
\usepackage{hypcap}
\usepackage{enumerate}
\usepackage{enumitem}
\usepackage[utf8]{inputenc}
\usepackage{physics}
\usepackage{cancel}
\usepackage{amssymb}
\usepackage{mathtools}
\usepackage{xcolor}
\usepackage{color}
\usepackage{framed}
\usepackage{amsthm}
\usepackage{blindtext}
\usepackage[toc,page]{appendix}
\usepackage{pgf,tikz,pgfplots}
\pgfplotsset{compat=1.15}
\usepackage{mathrsfs}
\pgfkeys{/pgf/declare function={asinh(\x) = ln(\x + sqrt(\x^2+1));}}
\usepackage{etoolbox}
\usetikzlibrary{external}
\usetikzlibrary{cd}
\AtBeginEnvironment{tikzcd}{\tikzexternaldisable}
\AtEndEnvironment{tikzcd}{\tikzexternalenable}
\usepackage{verbatim}
\usepackage{float}
\usepackage{placeins}
\usepackage{calc}
\newsavebox\CBox
\newcommand\hcancel[2][0.5pt]{%
  \ifmmode\sbox\CBox{$#2$}\else\sbox\CBox{#2}\fi%
  \makebox[0pt][l]{\usebox\CBox}%
  \rule[0.5\ht\CBox-#1/2]{\wd\CBox}{#1}}
\usepackage{url}
\usepackage{amsfonts,bm}
\usepackage{bbm}
\usepackage{colonequals}
\usepackage{pdfpages}

\usepackage[normalem]{ulem}
\usepackage{stackengine}
\usepackage{parskip} 
\numberwithin{equation}{section}

\def\@makechapterhead#1{%
  \vspace*{10\p@}%
  {\parindent \z@ \raggedright \sffamily
    \interlinepenalty\@M
    \Huge\bfseries \thechapter \space\space #1\par\nobreak
    \vskip 30\p@
  }}

\def\@makeschapterhead#1{%
  \vspace*{10\p@}%
  {\parindent \z@ \raggedright
    \sffamily
    \interlinepenalty\@M
    \Huge \bfseries  #1\par\nobreak
    \vskip 30\p@
  }}

  \makeatletter
  \def\subsection{\@startsection{subsection}{2}%
    \z@{.5\linespacing\@plus.7\linespacing}
  {.5\baselineskip}%
    {\normalfont\flushleft\scshape}%
  }
  \makeatother

  \makeatletter
  \def\subsubsection{\@startsection{subsubsection}{2}%
    \z@{.5\linespacing\@plus.7\linespacing}
  {.5\baselineskip}%
    {\normalfont\flushleft\scshape}%
  }
  \makeatother

  \usepackage{etoolbox}
  \patchcmd{\section}{\scshape}{\bfseries}{}{}
  \patchcmd{\subsection}{\scshape}{\bfseries}{}{}
  \patchcmd{\subsubsection}{\scshape}{\bfseries}{}{}
  \makeatletter
  \renewcommand{\@secnumfont}{\bfseries}
  \makeatother

\allowdisplaybreaks

\makeatletter
\DeclareRobustCommand\mapstofill{%
  $\m@th
  {\mapstochar}%
  \smash-\mkern-7mu
  \cleaders\hbox{$\mkern-2mu\smash-\mkern-2mu$}\hfill
  \mkern-7mu
  \mathord\rightarrow
  $%
}
\makeatother

\renewcommand{\div}{\text{div}}

\DeclareMathOperator{\arcsinh}{arcsinh}

\DeclareMathOperator{\sign}{sign}
\DeclareMathOperator{\diff}{d}
\DeclareMathOperator{\Tan}{T}

\DeclareMathOperator{\tracemine}{trace}

\DeclareMathOperator{\Ric}{Ric}

\theoremstyle{definition}
\newtheorem{deff}{Definition}[section]

\newtheorem{rem}[deff]{Remark}
\newtheorem{ex}[deff]{Example}

\theoremstyle{definition}
\newtheorem{prop}[deff]{Proposition}
\newtheorem{thmm}[deff]{Theorem}
\newtheorem*{thmm*}{Theorem}
\newtheorem{lemma}[deff]{Lemma}

\newtheoremstyle{named}{}{}{\itshape}{}{\bfseries}{.}{.5em}{\thmnote{#3's }#1}
\theoremstyle{named}

\newtheorem{claim}{Claim}

\theoremstyle{remark}
\newtheorem*{claimproof}{Proof of Claim}

\title[New Examples of Translating Solitons in GRW Geometries]{New Examples of Translating Solitons \\ in Generalised Robertson-Walker Geometries}
\author[Diego Artacho, Marie-Am\'elie Lawn, and Miguel Ortega]
{Diego Artacho$^{*}$, Marie-Am\'elie Lawn$^{*}$, and Miguel Ortega$^{\dagger}$
}

\address{Department of Mathematics, Imperial College London, London, SW7 2AZ, United Kingdom.}
\email{ d.artacho21@imperial.ac.uk, m.lawn@imperial.ac.uk}

\address{Institute of Mathematics, Department of Geometry and Topology, Faculty of Sciences, Campus de Fuentenueva, 18079 Granada (Spain).}
\email{miortega@ugr.es}

\allowdisplaybreaks

\begin{document}


\vspace{-2cm}
\maketitle
\vspace{-.5cm}
\begin{align*}
   \small {}^* \textit{Imperial College London},  \hspace{.2cm} {}^{\dagger} \textit{Universidad de Granada}
\end{align*}

\begin{abstract}
    \scriptsize
    Translators can be regarded as submanifolds which satisfy the mean curvature flow equation when evolving by translations along a distinguished vector field of the ambient space. We study translators in Generalised Robertson-Walker spacetimes, due to their importance as Lorentzian manifolds, and because they admit a natural conformal Killing timelike vector field carrying substantial geometric information, which will play the role of this translating vector field. We identify three one-parameter families of warping functions for which these objects exist. As a first example of this notion of translator, we classify the analogues of the classical Grim Reapers within this context.
\end{abstract}

{\scriptsize \tableofcontents}

\section{Introduction}\label{section:intro}

Since the seminal work of Huisken \cite{huisken}, solutions to the mean curvature flow in $\mathbb{R}^n $ have been extensively studied. Solitons, which are special solutions evolving by a group of conformal transformations of $\mathbb{R}^n$, have received particular attention \cite{Altschuler,chen,Clutterbuck,guang,notes,Smoczyk}. If the group is a subgroup of translations of the ambient space, these solutions are called translating solitons, or translators, and satisfy the simplified mean curvature flow equation
$$H=X^{\perp},$$
where the translating direction $X\in\mathbb{R}^{n}$ is a constant vector. Translating solitons are of great importance not only because they are eternal solutions to the mean curvature flow, but also because they appear naturally in the study of singularities \cite{Hamilton,singularities} and are intimately related to the theory of minimal surfaces. For example, it is well known that they are equivalent to minimal surfaces for a conformal metric \cite{ilmanenbook}. In recent years, many families of translators have been constructed using various techniques \cite{Clutterbuck,Smoczyk,Spruck}. Of particular interest are the translating solitons which are invariant under a group of isometries of the ambient space. In the case of translational or rotational invariance, they have been completely classified in \cite{Smoczyk} and \cite{Altschuler,Clutterbuck}. For the first case, the only non-trivial example was the Grim Reaper curve in the Euclidean plane. In higher dimensions, Grim Reapers are essentially the product of a Grim Reaper curve and a 2-codimensional hyperplane orthogonal to it.

Recently, generalisations of translating solitons to other ambient spaces have emerged. In \cite{deliramartin} the authors consider the general situation where the ambient manifold is a Riemannian product $(P\times \mathbb{R}, g_P+dt^2)$ of a Riemannian manifold with the real line. An obvious translating direction is given by the parallel vector field $\partial_t$, and one can naturally consider graphical solitons over $P$. Generalising this idea to Lorentzian products, the last two authors gave in \cite{LOSemiriemannian} new examples of translators in Minkowski space. In particular, they completely classified those translating solitons which are rotationally symmetric or, in other words, those invariant by the cohomogeneity-one action of the special orthogonal and orthochronal Lie groups. To do so, they use the general fact that the existence of such a cohomogeneity-one action ensures the manifold to admit a (pseudo-)Riemannian submersion to a one-dimensional base manifold, hence reducing the PDE to an ODE.

Going one step further, in \cite{deliramartin}, the authors actually consider a definition of translating solitons in the direction of an arbitrary vector field, although these solutions might not be solutions to the mean curvature flow in the classical sense. In this context, solutions associated with conformal Killing vector fields are obvious examples of study.

Our aim is to go further in this study, inspired by the translating solitons in a classical sense, but in a larger class of Lorentzian manifolds, using similar techniques to the ones in \cite{LOSemiriemannian,ALO25}. Instead of a (pseudo-)Riemannian product, we consider the well-established family of Generalised Robertson-Walker (GRW) spacetimes $M=N\times I$, with warped product metric $b(t) g_N-dt^2$, for $(N,g_N)$ a Riemannian manifold and $b:I\to (0,+\infty)$, with $n=\dim N\ge 1$, so $\dim M=n+1\ge 2$. Introduced in \cite{alias_romero_sanchez}, these spaces are of particular interest in Physics,  where they serve to construct famous models of the universe, including among others homogeneous, isotropic and expanding or contracting cases.  The vector field $b \, \partial_t$ is obviously a conformal Killing vector field, and we can consider translators in its direction. We also point out that  in \cite{delira} the authors study long-time existence of solutions to the mean curvature flow in another family of GRW spacetimes.

This paper is structured as follows: In Section 2, we recall some preliminary definitions and explain the setup. In Section 3, GRW spacetimes are introduced. We consider a function $u$ defined on an open subset of the base manifold $M$ and, by using the flow of $\partial_t$, we construct graphical hypersurfaces in our GRW spacetime. We ask this hypersurface to satisfy the equation
\begin{equation}\label{theorigin} (b \, \partial_t)^{\perp} = \vec{H} \, ,
\end{equation}
where $\vec{H}$ is the mean curvature vector of the hypersurface and $(b \, \partial_t)^{\perp}$ is the orthogonal projection of $b \, \partial_t$ to the normal bundle of the hypersurface. This is a new notion of translator. The corresponding PDE that $u$ must satisfy is obtained in Proposition \ref{prop:pde2}.

To make the theory meaningful, we then turn to obtaining some examples of this new notion of translator. As it is very common for this type of problem, we endow the space $N$ with a symmetry given by a codimension-$1$ Lie group action. As a consequence, the PDE reduces to an ODE, which we obtain in all generality in Proposition \ref{prop:finalode}. As a first example, we focus on the analogue of the famous classical Grim Reaper family within this setting. This means we take $N= \mathbb{R}^{n+1}$ endowed with the action of $\mathbb{R}^n$ by translation of the first $n$ coordinates.

We find three families of warping functions $b$ for which these translators exist. We point out that in \cite{delira}, the authors needed the Null Convergence Condition in their study. In contrast, we are exhibiting some examples where this condition does not hold.

In Section \ref{section:grimreapers}, we classify the Grim Reapers in their corresponding GRW spacetimes -- \textit{vid.} Theorem \ref{thm:maintheorem}. These are the first families of examples of our new notion of translator. We anticipate that these methods will yield numerous additional examples in subsequent studies, contributing to the further development and understanding of this new framework.

\section{Preliminaries}\label{section:preliminaries}


Let $\left(M,g_M\right)$ and $\left(N^n,g_N\right)$ be pseudo-Riemannian manifolds. Let $\varepsilon>0$. A smooth family of smooth embeddings $\left\{F_t \colon \left(N,g_N\right) \to \left(M,g_M\right) \colon t \in (-\varepsilon, \varepsilon)\right\}$ is said to be a solution to the mean curvature flow if, and only if, for all $t\in (-\varepsilon, \varepsilon)$, the pullback $F_t^* g_M$ is non-degenerate and
\begin{equation*}
  \left(\frac{\partial F_t}{\partial t}\right)^{\perp} = \vec{H}_t \, ,
\end{equation*}
where $\vec{H}_t$ is the mean curvature vector of $F_t$. Recall that, if $\left( e_1 , \dots , e_n \right)$ is a local $(F_{t}^{*}g_M)$-orthonormal frame of $N$, with $g\left(\left( F_t \right)_* e_i,\left( F_t \right)_* e_j\right)=\varepsilon_i \delta_{ij}$, for $i,j=1,\dots,n$, then $\vec{H_t}$ is defined as
\begin{equation}\label{eq:def_h}
\vec{H}_t = \sum_{i=1}^{n} \varepsilon_i \left( \nabla^{M}_{\left(F_t\right)_* e_i} \left(F_t\right)_*e_i \right)^{\perp} \, ,
\end{equation}
 where $\nabla^M$ is the Levi-Civita connection of $(M,g_M)$ -- along the paper, we will denote by $\nabla^P$ the Levi-Civita connection of a  given pseudo-Riemannian manifold $(P,g_P)$ -- and
$\perp$ denotes orthogonal projection to the normal bundle of $F_t$. Now suppose we have a smooth family of conformal maps $\left\{C_t \colon \left(M,g_M\right) \to \left(M,g_M\right)\colon t \in (-\varepsilon,\varepsilon)\right\}$ and a fixed pseudo-Riemannian immersion $F_0 \colon \left(N,g_N\right) \to \left(M,g_M\right)$. We say that $F_0$ defines a \textbf{soliton} with respect to the family $\{C_t \colon t \in (-\varepsilon,\varepsilon) \}$ if, and only if, the family $\left\{F_t = C_t \circ F_0 \colon t \in (-\varepsilon,\varepsilon) \right\}$ is a solution to the mean curvature flow.

  Suppose that $M = N \times I$, with $I \subseteq \mathbb{R}$ an open interval, and that $g_M (x,t) = \varphi (x,t) g_N (x) + \psi (x,t) dt^2$ is a pseudo-Riemannian metric on $M$, for some smooth functions $\varphi,\psi \colon N \times I \to \mathbb{R}$. These spaces provide us with:

  \begin{enumerate}
    \item a natural direction in which translate, namely the $I$-direction;
    \item a natural class of candidates to solitons, namely graphs of functions $N \to I$.
  \end{enumerate}

  For some choices of the functions $\varphi$ and $\psi$, one can find a conformal Killing vector field on $M$, whose flow will provide us with a one-parameter family of conformal maps with respect to which one can try to find functions $N \to I$ whose graph is a soliton. In this sense, we are generalising the classical notion of translating soliton \cite{deliramartin}.

  In \cite{LOSemiriemannian}, Lawn and Ortega studied the case where $\varphi$ is constant equal to $1$ and $\psi$ is also constant equal to $\pm 1$. We will now look at other choices of $\varphi$ and $\psi$ which have interest in other areas of Mathematics and Physics.

\section{Generalised Robertson-Walker Geometries}  \label{section:rw}


We now turn our attention to Generalised Robertson-Walker geometries, which in our setting correspond to taking $\psi$ constant equal to $-1$ and $\varphi = \tilde{\varphi} \circ \pi_2$, where $\tilde{\varphi} \colon I \to (0,\infty)$ is smooth and $\pi_2$ denotes the projection to the second factor. These spaces were introduced by Al\'ias, Romero and S\'anchez in \cite{alias_romero_sanchez}, and have been widely studied \cite{Mant,Lawn_2015}.

Let $\left( N, g_N \right)$ be an $n$-dimensional Riemannian manifold, $I \subseteq \mathbb{R}$ an open interval and $b \colon I \to \left(0,+\infty\right)$ a smooth function. A Generalised Robertson-Walker spacetime (GRW) with base $I$, fibre $\left( N, g_N \right)$ and warping function $b$ is the product manifold $M=N \times I$ endowed with the metric
\begin{equation*}
    g = g_{M} = b(t)^2g_N - dt^2 \, ,
\end{equation*}
where $t$ is the coordinate in $I$. We will denote this space by $\mathcal{RW} \left( N,g_N,I,b \right)$.

When $N$ is a real spaceform, these spaces are used in Physics to model a homogeneous, isotropic and expanding or contracting universe, \textit{vid.} \cite{Mant} and references therein.

Following the previously discussed ideas, these spaces come naturally equipped with a conformal Killing vector field.

\begin{lemma} \label{lemma:conformal}
The vector field $X \coloneqq b(t)\partial_t$ is a conformal Killing vector field on $\left(M,g \right) = \mathcal{RW} \left(N,g_N,I,b\right)$. More specifically,
\begin{equation*}
\pushQED{\qed}
\mathcal{L}_X g = 2 b'g . \qedhere
\popQED
\end{equation*}
\end{lemma}

Lemma \ref{lemma:conformal} provides us with a one-parameter family of conformal diffeomorphisms between open subsets of $M$, namely the flow $\Phi_s$ of the vector field $X=b(t) \partial_t$. Note that, for all $x \in N$ and $t$ sufficiently small,

\begin{equation*}
\Phi_t (x,s) = (x,A(s,t)) ,
\end{equation*}
where $A(s,t)$ satisfies

\begin{equation} \label{eq:A}
     \left\{
\def\arraystretch{1.5}
\begin{array}{ll}
      A(s,0) = s ,  \\
      \frac{\partial A}{\partial t}(s,t) = b\left( A(s,t) \right) . \\
\end{array}
\right.
\end{equation}

We will restrict our attention to hypersurfaces of $M$ which are graphs of functions $N \to I$, as this is the generic case. Let $\Omega \subseteq N$ be open, precompact and connected, and let $u \in C^2(\Omega,I)$. Then, there exists $\varepsilon > 0$ such that the flow $\Phi_t$ of $X$ is defined in $\Omega$ for all $t \in (-\varepsilon, \varepsilon)$. For $t\in (-\varepsilon, \varepsilon)$, consider
\begin{align*}
    \begin{split}
F_t \colon \Omega &\rightarrow N \times I = M\\
    x &\mapsto \Phi_t\left(x,u(x) \right) = (x,u_t(x))
    \end{split} \hspace{.2cm} ,
\end{align*}
where $u_t (x) \coloneqq A\left( u(x) , t \right)$. $F_t$ is an embedding for every $t$, as it is a graph map. Suppose further that $F_0^*g$ is non-degenerate -- and hence $F_t^*g$, for all sufficiently small $t$. In this setting, we can establish an important definition:

\begin{deff}
    We say that $F \coloneqq F_0$ is an $X$-soliton if, and only if, for every $x \in \Omega$ and $t \in (-\varepsilon, \varepsilon)$,
\begin{equation} \label{eq:soliton_rw}
    \left( b \circ u_t \right) (x) \, \partial_{t}^{\perp} (x,u_t(x)) = \vec{H_t} (x, u_t(x)) \, ,
\end{equation}
where $\vec{H_t}$ is the mean curvature of the pseudo-Riemannian immersion $F_t$. As in \cite{deliramartin}, we say that $F$ defines a \textit{pointwise} $X$-soliton if \eqref{eq:soliton_rw} is satisfied for $t=0$.
\end{deff}

In \cite{delira}, de Lira and Roing considered some natural conditions for the longtime existence of the mean curvature flow in a class of GRW spacetimes. In Section \ref{section:types} we will introduce another class of such spacetimes which 
does not satisfy all their hypotheses.

\subsection{Characterising PDE}

In this section we will obtain the PDE characterising the property of being a (pointwise) $X$-soliton in a GRW spacetime.

\begin{rem}\label{rem:simplicity}
    For the moment, for simplicity, we will work only with the graph map $F$ of a function $u \in \mathcal{C}^2(\Omega,I)$, where $\Omega \subseteq N$ is open, precompact and connected, to find a more explicit expression for equation \eqref{eq:soliton_rw}. We will find a PDE for $u$ characterising the property of $F$ being a pointwise $X$-soliton. The $X$-soliton case will result from substituting $u$ with $u_t$ and $F$ with $F_t$, and imposing that the PDE be satisfied for all $t$ in a suitable neighbourhood of $0$.
\end{rem}
Note that, under the usual identifications, if $X \in \Gamma \left(\Tan N \right)$, then
\begin{equation*}
F_*X = \left(X,0\right) + g_N\left( \nabla u , X \right)\partial_t = \left(X,g_N\left( \nabla u , X \right) \right)\, .
\end{equation*}
Also, the vector field
\begin{equation*}
    \frac{1}{b(u(x))^2}\left(\nabla u (x) , 0 \right) + \partial_t
\end{equation*}
is normal to $F$.

\begin{lemma}
Under the previous notation and hypotheses, for all $x$ in $\Omega$ the quantity
\begin{equation*}
\hat{\varepsilon} (x) := \sign\left( \frac{\abs{\nabla u (x)}^2}{b(u(x))^2} -1 \right)\in\{\pm 1\} \, ,
\end{equation*}
and it is constant.
\end{lemma}

\begin{proof}

Indeed, suppose it is not. Then, as $\Omega$ is connected, there exists $x \in \Omega$ such that
\begin{equation} \label{absurd}
     \frac{\abs{\nabla u (x)}^2}{b(u(x))^2} -1 = 0 .
\end{equation}

In such case, we would have for every tangent vector $Y \in \Tan_x N$ that
\begin{align*}
    \left( F^* g \right) \left( \nabla u (x) , Y \right) &= g \left( F_* \nabla u (x) , F_* Y \right) \\
    &= g \left( \left(\nabla u (x),0\right) + \abs{\nabla u (x)}^2 \partial_t , Y + g_N \left(\nabla u (x) , Y \right)\partial_t \right) \\
    &= b(u(x))^2 g_N \left( \nabla u (x), Y \right) - \abs{\nabla u (x)}^2 g_N \left( \nabla u (x), Y \right) \\
    &= 0 .
\end{align*}

But, by hypothesis, $F^*g$ is non-degenerate, so this would imply that $\nabla u (x) = 0$, which in turn would imply that $\abs{\nabla u (x)}^2 = 0$, which contradicts \eqref{absurd}.
\end{proof}

So the vector field
\begin{equation} \label{eq:definitionnu}
\nu (x)= \frac{1}{W(x)}\left( \frac{1}{b(u(x))^2}\left(\nabla u (x) , 0\right) + \partial_t \right) , \, x \in \Omega \, ,
\end{equation}
where
\begin{equation} \label{eq:definitionW}
    W (x) = \sqrt{\hat{\varepsilon} (x) \left( \frac{\abs{\nabla u (x)}^2}{b(u(x))^2} -1 \right)} \, ,
\end{equation}
is a normal vector field to $F$ with $g(\nu,\nu) = \hat{\varepsilon}$.

In \cite{rw_curvature}, Hernandes de Lima computed the mean curvature vector of $F$, obtaining
\begin{equation}\label{eq:lima}
    g \left( \vec{H}, \nu \right) = -\div \left( \frac{\nabla u}{(b\circ u)^2 W}\right) - \frac{(b'\circ u)}{(b\circ u)W}\left( n + \frac{\abs{\nabla u}^2}{(b\circ u)^2}\right) ,
\end{equation}
where $n = \dim(N)$.

\begin{thmm} \label{thm:generalpde}
    Under the previous hypotheses and notation, $F$ is a pointwise $X$-soliton if, and only if, $u$ satisfies the following partial differential equation:
    \begin{equation} \label{eq:pde1}
        \div \left( \frac{\nabla u}{\left(b \circ u\right) \, W}\right) = \frac{1}{W}\left[ \left(b \circ u\right)^2 - n \left(b' \circ u \right) \right].
    \end{equation}
    \end{thmm}

    \begin{proof}

    $\vec{H} = (b \circ u) \partial_t^{\perp} \text{  if and only if  } g \left( \vec{H},\nu \right) = (b \circ u) g \left( \partial_t^{\perp} , \nu \right) = (b \circ u) g \left( \partial_t , \nu \right) = -(b\circ u) / W$, as $\left( \nabla u , 0 \right) \perp \partial_t$. But, using equation \eqref{eq:lima},
    \begin{align*}
        &g \left( \vec{H}, \nu \right) = -\div \left( \frac{\nabla u}{(b\circ u)^2 W}\right) - \frac{(b'\circ u)}{(b\circ u)W}\left( n + \frac{\abs{\nabla u}^2}{(b\circ u)^2}\right) \\
        &= \frac{-1}{(b \circ u)} \div \left( \frac{\nabla u}{(b\circ u) W}\right) - g_N \left( \nabla \left( \frac{1}{(b \circ u)} \right) , \frac{\nabla u}{(b \circ u) W} \right) - \\
        & - \frac{(b'\circ u)}{(b\circ u)W}\left( n + \frac{1}{(b\circ u)^2}\abs{\nabla u}^2\right) \\
        &= \frac{-1}{(b \circ u)} \left( \div \left( \frac{\nabla u}{(b \circ u) W} \right) + \frac{n (b' \circ u)}{W}\right) ,
    \end{align*}
    which yields the result.
    \end{proof}

    As explained in Remark \ref{rem:simplicity}, we obtain the characterisation of $X$-solitons from pointwise $X$-solitons by substituting $u_t$ for $u$ in \eqref{eq:pde1}, obtaining the following:

    \begin{prop} \label{prop:pde2}
         The family $\{F_t \colon t \in (-\varepsilon,\varepsilon)\}$ defines an $X$-soliton if, and only if, for each $x \in \Omega$ and $t \in (-\varepsilon,\varepsilon)$:
        \begin{equation} \label{eq:pde2}
            \div \left( \frac{\nabla u}{\left(b \circ u\right) \, W}\right) (x) = \frac{1}{W(x)}\left[ \left(b(u_t(x))\right)^2 - n \left(b'(u_t(x)) \right) \right] \, .
            \end{equation}
        \end{prop}

        \begin{proof}
        Note that, by the chain rule,
        \begin{equation*}
        \nabla u_t = \nabla \left( A \left( u(x),t \right) \right) = \partial_1 A \left( u(x),t \right) \nabla u \, ,
        \end{equation*}
        where $\partial_1$ denotes partial derivative with respect to the first variable. Let us compute $\partial_1 A \left( u(x),t \right)$. By definition, $A$ satisfies \eqref{eq:A}. Hence, for every $t,s$,
        \begin{equation*}
        \int_{A(s,0)}^{A(s,t)} \frac{1}{b(y)}dy = \int_{0}^{t} dr = t \, .
        \end{equation*}
        Thus, by Leibniz's integral rule,
        \begin{align*}
        0 &= \frac{\partial}{\partial s} t = \frac{\partial}{\partial s} \left( \int_{A(s,0)}^{A(s,t)} \frac{1}{b(y)}dy  \right) = \frac{1}{b\left( A(s,t) \right)} \frac{\partial A}{\partial s} (s,t) - \frac{1}{b\left( A(s,0) \right)} \frac{\partial A}{\partial s}(s,0) \, ,
        \end{align*}
        and hence
        \begin{equation*}
            \frac{\partial A}{\partial s}(s,t) = \frac{b\left( A(s,t) \right)}{b(s)} .
        \end{equation*}
        So we have that
        \begin{equation*}
        \nabla u_t = \frac{b \circ u_t}{b \circ u} \nabla u \, .
        \end{equation*}
        Using this, a straightforward computation shows that $W_t = W$ for every $t$, and the result follows.
        \end{proof}

        \begin{rem}
        If $b = 1$, equation \eqref{eq:pde2} coincides with the equation found by Lawn and Ortega in \cite{LOSemiriemannian}.
        \end{rem}

\subsection{Reduction to an ODE}\label{section:submersions}

We refer to \cite{Besse} and \cite{oneillpaper} for a complete treatment of pseudo-Riemannian submersions, but we recall some basic facts here. Let $\left( N, g_N \right)$ and $\left( B, g_B \right)$ be pseudo-Riemannian manifolds. A submersion $\pi \colon N \to B$ is said to be a pseudo-Riemannian submersion if it is surjective, the fibres are pseudo-Riemannian submanifolds, and its derivative $\pi_*$ preserves lengths of horizontal vectors. If $Y \in \Gamma(\Tan B)$, we denote by $Y^h$ its horizontal lift. If $Z \in \Gamma(\Tan N)$, we denote by $\mathcal{H}(Z)$ (resp. $\mathcal{V}(Z)$) its horizontal (resp. vertical) component.

Note that, if a submersion $\pi \colon \left( N, g_N \right) \to \left( B, g_B \right)$ is pseudo-Riemannian, then the submersion $\tilde{\pi} \colon \left( N \times I, g_{N \times I} = b(t)^2 g_N - dt^2 \right) \to \left( B \times I, g_{B \times I} = b(t)^2 g_B - dt^2 \right)$ defined by $(x,t) \mapsto (\pi(x),t)$ is also pseudo-Riemannian. And the $\tilde{\pi}$-horizontal lift of $\partial_t$ is $\partial_t$.

Now take $\left( e_1 , \dots , e_n \right)$ a local $g_N$-orthonormal frame of $N$, with $g_N (e_i,e_j) = \varepsilon_i \delta_{ij}$, for $i,j = 1, \dots, n$, with $\varepsilon_i \in \{\pm1\}$ for each $i = 1, \dots, n$. We can take $e_1, \dots , e_k$ to be $\pi$-vertical and $e_{k+1}, \dots , e_n$ to be $\pi$-horizontal.

Let $U \subseteq B$ be an open subset of $B$ and let $f \in \mathcal{C}^2 \left(U,I\right)$. Define $u \coloneqq f \circ \pi$. Denote by $F$ and $\tilde{F}$ the graph maps of $f$ and $u$ respectively.

\[
    \begin{tikzcd}[row sep=.5cm, column sep=large]
        & \Omega = \pi^{-1}\left(U\right) \subseteq N \arrow{r}{\tilde{F} = id \times u} \arrow{dd}[swap]{\pi} \arrow{dl}{u} & N \times I \arrow{dd}{\tilde{\pi} = \pi \times id} \\
        I & &  \\
        & U \subseteq B \arrow{r}{F = id \times f} \arrow{ul}{f} & B \times I
        \end{tikzcd}
\]

\begin{lemma} \label{lemma:basics}
In the previous situation, we have the following:
\begin{enumerate}
    \item $W_{\tilde{F}} = W_F \circ \pi$.
    \item The $\tilde{\pi}$-horizontal lift of $\nu_{F}$ is $\nu_{\tilde{F}}$.
    \item For every $k+1\leq i,j \leq n$ the $\tilde{\pi}$-horizontal lift of $F_* \pi_* e_i$ is $\tilde{F}_* e_i$.
    \item For every $k+1 \leq i,j \leq n$,
    \begin{equation*}
    \tilde{\pi}_* \nabla^{N \times I}_{\tilde{F}_* e_i} \tilde{F}_* e_j = \nabla^{B \times I}_{F_* \pi_* e_i} F_* \pi_* e_j \, .
    \end{equation*}
\end{enumerate}
\end{lemma}

\begin{proof}
Note that, for every $k+1 \leq i \leq n$, \textit{i.e.}, when $e_i$ is horizontal, we have that
\begin{equation*}
    g_N \left(\left[\nabla f \right]^h, e_i \right) = \pi^* g_B \left(\nabla f, \pi_* e_i \right) = \pi^*df \left(\pi_* e_i\right) =g_N \left(\nabla \left(f \circ \pi \right) , e_i \right) .
\end{equation*}
Thus, $ \left[\nabla f\right]^h = \nabla (f \circ \pi) $. From this, we deduce the first three parts of this Lemma. For the last point, it is a general fact that, if $p \colon \mathcal{N} \to \mathcal{B}$ is a pseudo-Riemannian submersion, then for all $X,Y$ vector fields on $\mathcal{B}$
\begin{equation*}
    \mathcal{H} \left( \nabla^{\mathcal{N}}_{X^h} Y^h \right) = \left( \nabla^{\mathcal{B}}_{X} Y \right)^h ,
\end{equation*}
as can be found in \cite{oneill}.
\end{proof}

We now want to link the mean curvature of a graphical submanifold in $\mathcal{RW}\left(B,g_B,I,b \right)$ with the mean curvature vector of its pullback by $\tilde{\pi}$. The first step is to relate the corresponding metrics.

\begin{lemma} \label{lemma:metric}
Let $\tilde{\gamma} = \tilde{F}^*g_{N \times I}$ and $\gamma = F^*g_{B \times I}$. Then, for all $x \in \Omega \subseteq N$:
\begin{equation*}
    \tilde{\gamma}_{ij}(x) =
    \begin{dcases}
    b(f(\pi(x)))^2 \varepsilon_i \delta_{ij} & 1 \leq i \leq k \hspace{.5cm} \text{or} \hspace{.5cm} 1 \leq j \leq k\\
    \gamma_{ij}(\pi(x)) & k+1 \leq i,j \leq n
 \end{dcases} \, .
\end{equation*}

Consequently,
\begin{equation*}
    \tilde{\gamma}^{ij}(x) =
    \begin{dcases}
    \frac{1}{b(f(\pi(x)))^2} \varepsilon_i \delta^{ij} & 1 \leq i \leq k \hspace{.5cm} \text{or} \hspace{.5cm} 1 \leq j \leq k\\
    \gamma^{ij}(\pi(x)) & k+1 \leq i,j \leq n
 \end{dcases} \, .
\end{equation*}
\end{lemma}

\begin{proof}
It is immediate from the definitions.
\end{proof}

Our goal is to reduce the partial differential equation \eqref{eq:pde2} to an ordinary differential equation. To do this, we will take the base $B$ of the submersion to be $1$-dimensional. Hence, suppose from now on that we have a pseudo-Riemannian submersion $\pi \colon \left( N, g_N \right) \to \left( J, g_J \right)$, where $J \subseteq \mathbb{R}$ is an open interval with a general pseudo-Riemannian metric $g_J (s) = \tilde{\varepsilon} \alpha(s) \diff s^2$, $s \in J$, where $\tilde{\varepsilon} \in \{\pm 1\}$ and $\alpha \colon J \to (0,+\infty)$ is a strictly positive smooth function. We will assume that the fibres of $\pi$ have constant mean curvature, so that we get a smooth function $h \colon J \to \mathbb{R}$ with $h(s)$ being the mean curvature of the hypersurface $\pi^{-1}\{s\}$ with respect to the normal vector field $\nabla \pi$, for each $s \in J$.
Now note that
\begin{equation*}
\nabla f = \frac{\tilde{\varepsilon}f'}{\alpha} \partial_s \hspace{1cm} \text{and hence} \hspace{1cm} g_J \left( \nabla f , \nabla f \right) = \frac{\tilde{\varepsilon}f'^2}{\alpha} .
\end{equation*}
So, in the previous notation,
\begin{equation}\label{eq:defW}
W \coloneqq W_F = \sqrt{\hat{\varepsilon} \left( \frac{\tilde{\varepsilon}f'^2}{(b\circ f)^2 \alpha} -1 \right)} \, .
\end{equation}
Now we need a technical Lemma.

\begin{lemma}\label{lemma:div}
Let $z \colon J \to \mathbb{R}$ be a smooth function. Then,
\begin{equation*}
\div_J \left( z(s) \partial_s \right) = z' + z \frac{\alpha'}{2\alpha} \, .
\end{equation*}
\end{lemma}

\begin{proof}
Indeed,
\begin{align*}
    \pushQED{\qed}
    \div_J \left( z(s) \partial_s \right) &= \tilde{\varepsilon} g_J \left( \nabla^{J}_{\frac{1}{\sqrt{\alpha}}\partial_s} \left( z(s) \partial_s \right) , \frac{1}{\sqrt{\alpha}} \partial_s \right) \\
    &= \frac{\tilde{\varepsilon}}{\alpha} g_J \left( z'(s) \partial_s + z(s) \nabla^{J}_{\partial_s} \partial_s , \partial_s \right) \\
    &= z' + z \frac{\alpha'}{2\alpha} \, . \qedhere
    \popQED
\end{align*}
\renewcommand{\qed}{}
\end{proof}

Finally, we compute the mean curvature vector of $\tilde{F}$ in terms of the mean curvature vector of $F$.

\begin{lemma}\label{lemma:curvature}
In the previous situation, we have that:
\begin{equation*}
\vec{H}_{\tilde{F}} = \vec{H}_{F}^h + \hat{\varepsilon} \left( \frac{h(f' \circ \pi)}{(b\circ f \circ \pi)^2(W\circ \pi)} - \frac{(b'\circ f \circ \pi)}{(b\circ f \circ \pi)(W \circ \pi)} (n-1)\right) \nu_{\tilde{F}} \, .
\end{equation*}
\end{lemma}

\begin{proof}
Recall \eqref{eq:def_h}. In our setting, for each $X,Y \in \Gamma\left( \Tan N \right)$ the second fundamental form of $\tilde{F}$ is defined by
\begin{equation*}
II_{\tilde{F}} (X,Y) = \left( \nabla^{N \times I}_{\tilde{F}_* X} \left(\tilde{F}_* Y \right) \right)^{\perp} \, ,
\end{equation*}
and
\begin{equation*}
\vec{H}_{\tilde{F}} = \tracemine_{\tilde{\gamma}} II_{\tilde{F}} = \sum_{i,j=1}^{n} \tilde{\gamma}^{ij} \left( \nabla^{N \times I}_{\tilde{F}_* e_i} \left( \tilde{F}_* e_j \right) \right)^{\perp} = \hat{\varepsilon} \sum_{i,j=1}^{n} \tilde{\gamma}^{ij} g_{N \times I} \left( \nabla^{N \times I}_{\tilde{F}_* e_i} \tilde{F}_* e_j, \nu_{\tilde{F}} \right) \nu_{\tilde{F}} \, .
\end{equation*}
    Using Lemmata \ref{lemma:basics} and \ref{lemma:metric}, we obtain:
\begin{align*}
\vec{H}_{\tilde{F}} &= \hat{\varepsilon} \sum_{i,j=1}^{n} \tilde{\gamma}^{ij} g_{N \times I} \left( \nabla^{N \times I}_{\tilde{F}_* e_i} \tilde{F}_* e_j, \nu_{\tilde{F}} \right) \nu_{\tilde{F}}\\
&= \hat{\varepsilon} \sum_{i=1}^{k}\sum_{j=1}^{k} \frac{1}{b^2} \varepsilon_i \delta_{ij} g_{N \times I} \left( \nabla^{N \times I}_{\tilde{F}_* e_i} \tilde{F}_* e_j, \nu_{\tilde{F}} \right) \nu_{\tilde{F}} + \\
&+ \hat{\varepsilon} \sum_{i=k+1}^{n}\sum_{j=k+1}^{n} \left(\gamma^{ij} \circ \pi \right) g_{N \times I} \left( \nabla^{N \times I}_{\tilde{F}_* e_i} \tilde{F}_* e_j, \nu_{\tilde{F}} \right) \nu_{\tilde{F}} \\
&= \frac{1}{(b\circ f \circ \pi)^2} \hat{\varepsilon} \sum_{i=1}^{k} \varepsilon_i g_{N \times I} \left( \nabla^{N \times I}_{\tilde{F}_* e_i} \tilde{F}_* e_i, \nu_{\tilde{F}} \right) \nu_{\tilde{F}} + \vec{H}_{F}^{h} \, .
\end{align*}

But, for every $1 \leq i \leq k$, using that $e_i$ is $\pi$-vertical, we obtain that
\begin{equation*}
    g_{N \times I} \left( \nabla^{N \times I}_{\tilde{F}_* e_i} \tilde{F}_* e_i, \nu_{\tilde{F}} \right) = \frac{-(b\circ f \circ \pi) (b'\circ f \circ \pi)}{W\circ \pi} \varepsilon_i + \frac{1}{W \circ \pi} g_N \left( \nabla^{N}_{e_i} e_i, \nabla u\right) .
\end{equation*}

In our case, $k = n-1$, and $\nabla u = \left( f' \circ \pi \right) \nabla \pi$, which yields the result.
\end{proof}

And we obtain the desired ODE:

\begin{thmm}
In the previous situation, $\tilde{F}$ defines a pointwise $X$-soliton if, and only if, $f$ satisfies the following ODE:
\begin{equation} \label{ode}
    f'' = \left( \tilde{\varepsilon} \alpha - \frac{f'^2}{\left( b \circ f \right)^2}\right) \left( h f' + \left( b \circ f \right) \left[ \left( b \circ f\right)^2 - n \left( b' \circ f \right) \right]  \right) + \frac{f'}{2} \left( \log \left( \alpha [b \circ f]^2 \right) \right)' \, .
\end{equation}
\end{thmm}

\begin{proof}
Note that $\vec{H}_{\tilde{F}} = (b \circ f \circ \pi) \, \partial_{t}^{\perp}$ if, and only if,
\begin{align*}
 g_{N \times I} \left( \vec{H}_{\tilde{F}} , \nu_{\tilde{F}} \right) &=  (b \circ f \circ \pi) \, g_{N \times I} \left( \partial_{t}^{\perp} , \nu_{\tilde{F}} \right) \\
&= (b\circ f \circ \pi) \, g_{N \times I} \left( \partial_{t} , \nu_{\tilde{F}} \right) \\
&= \frac{-(b\circ f \circ \pi)}{W \circ \pi} \, .
\end{align*}
By Lemma \ref{lemma:curvature}, this is equivalent to
\begin{equation*}
g_{N \times I} \left( \vec{H}_{F}^h , \nu_{F}^h \right) + \frac{h(f'\circ \pi)}{(b\circ f \circ \pi)^2W} - \frac{b'\circ f \circ \pi}{(b\circ f \circ \pi)(W \circ \pi)}(n-1) = \frac{-(b \circ f \circ \pi)}{W \circ \pi} \, .
\end{equation*}
And this is equivalent to
\begin{equation*}
\pi^* \left(g_{J \times I} \left( \vec{H}_{F} , \nu_{F} \right) \right) = \frac{1}{W\circ \pi} \left[ \frac{b' \circ f \circ \pi}{b \circ f \circ \pi }(n-1) - \frac{h(f'\circ \pi)}{(b \circ f \circ \pi)^2} - (b \circ f \circ \pi) \right] \, .
\end{equation*}

But now, using \eqref{eq:lima}, the product rule for the divergence, Lemma \ref{lemma:div} and the definition of $W$ \eqref{eq:defW}, we get:
\begin{align*}
    & g_{J \times I} \left( \vec{H}_{F} , \nu_{F} \right) = \frac{-1}{(b\circ f)^2} \div \left( \frac{\nabla f}{W}\right) + \frac{b' \circ f}{(b\circ f)W} \left( \frac{1}{(b\circ f)^2}\lvert \nabla f \rvert^2 -1 \right) \\
    &= \frac{-1}{(b\circ f)^2} \div \left( \frac{\nabla f}{W}\right) + \frac{\hat{\varepsilon} (b' \circ f)}{b\circ f}W \\
    &= \frac{-1}{(b\circ f)^2} \div \left( \frac{\tilde{\varepsilon}f'}{\alpha W} \partial_s \right) + \frac{\hat{\varepsilon} (b' \circ f)}{b \circ f}W \\
    &= \frac{-1}{(b\circ f)^2} \left( \frac{\tilde{\varepsilon}f'}{\alpha W} \right)' - \frac{1}{(b\circ f)^2} \frac{\tilde{\varepsilon}f'}{\alpha W} \frac{\alpha'}{2\alpha} + \frac{\hat{\varepsilon} (b'\circ f)}{b \circ f}W \\
    &= \frac{\hat{\varepsilon} \tilde{\varepsilon}}{(b\circ f)^2 \alpha W^3}f'' - \frac{\hat{\varepsilon} \tilde{\varepsilon} \alpha'}{2(b\circ f)^2 \alpha^2 W^3}f' - \frac{\hat{\varepsilon} (b' \circ f)}{(b\circ f)^5 \alpha^2 W^3}f'^4 + \frac{\hat{\varepsilon} (b'\circ f)}{b \circ f}W \, .
\end{align*}

Putting everything together, after a tedious but straightforward computation, one obtains the result.
\end{proof}

\begin{prop}\label{prop:finalode}
In the situation of the beginning of this section, $F$ defines a (pointwise) $X$-soliton if, and only if, for all $t$ (for $t=0$),
\begin{equation} \label{eq:ode2}
    f'' = \left( \tilde{\varepsilon} \alpha - \frac{f'^2}{\left( b \circ f \right)^2}\right) \left( h f' + \left( b \circ f \right) \left[ \left( b \circ f_t\right)^2 - n \left( b' \circ f_t \right) \right]  \right) + \frac{f'}{2} \left( \log \left( \alpha [b \circ f]^2 \right) \right)' \, .
    \end{equation}
\end{prop}

\begin{proof}
It is analogous to the one of Proposition \ref{prop:pde2}. One first finds that, for each $t$,
\begin{equation*}
f_t ' = \frac{b \circ f_t}{b \circ f} \, f' \, ,
\end{equation*}
and
\begin{equation*}
f_t '' = \frac{b \circ f_t}{b \circ f} \, f'' + \frac{b \circ f_t}{(b \circ f)^2} \, f'^2 \left( (b' \circ f_t)-(b' \circ f) \right) \, .
\end{equation*}
Then, by substituting $f_t$ for $f$ in equation \eqref{ode}, one gets the result.
\end{proof}

We are interested in solutions $f$ to the ODE \eqref{eq:ode2} which do not depend explicitly on $t$. Suppose that $\tilde{\varepsilon}=+1$, \textit{i.e.}, that the codomain of $\pi$ is Riemannian (recall that $\alpha >0$). Then, one can easily find two families of solutions to the equation
\begin{equation}\label{eq:miracle}
    \alpha - \frac{f'^2}{\left( b \circ f \right)^2} = 0
\end{equation}
by realising that
\begin{equation*}
    \alpha - \frac{f'^2}{\left( b \circ f \right)^2} = \left( \sqrt{\alpha} + \frac{f'}{\left( b \circ f \right)} \right) \left( \sqrt{\alpha} - \frac{f'}{\left( b \circ f \right)} \right),
\end{equation*}
and using the method of separation of variables.

Any function $f$ satisfying \eqref{eq:miracle} also satisfies
\begin{equation*}
    f'' =  \frac{f'}{2} \left( \log \left( \alpha [b \circ f]^2 \right) \right)' \, ,
\end{equation*}
as can be readily verified. This fact will be crucial for our discussion in Section \ref{section:grimreapers}.

These solutions satisfy $W = 0$, so the restriction of the metric to these surfaces is degenerate, and hence the notion of mean curvature does not make sense on them. They define lightlike submanifolds in $M$, which will be of great importance in what follows.

\subsection{Special families of warping functions}\label{section:types}

We now turn to the study of the solutions to \eqref{eq:pde2}. For general warping functions, the general solution to \eqref{eq:pde2} will depend explicitly on the parameter $t$. This is reasonable, since we are imposing that, when we move the graph of a function according to the flow of a vector field, we have the soliton condition \eqref{eq:soliton_rw} for every $t$. This is quite a strong condition, and it makes sense for it not to happen in the general case. However, by the nature of the problem, solutions to \eqref{eq:pde2} with geometric significance are the ones which do not depend explicitly on $t$.

For some particular choices of warping function $b$, we can get rid of the time dependence of the general solution to \eqref{eq:pde2}. Let us see this.

Note that the dependence on $t$ in the PDE \eqref{eq:pde2} and in the ODE \eqref{eq:ode2} is condensed in the term
\begin{equation*}
\left( b \circ u_t \right)^2 - n \left( b' \circ u_t \right) .
\end{equation*}

A \textit{sufficient} condition for the general solution not to depend explicitly on $t$ is that this term be constant, which we will call $d$. A sufficient condition for this to happen is that
\begin{equation}\label{eq:condition_b}
2 b b' - n b'' = 0 .
\end{equation}

This forces $b \colon I \to \left( 0, +\infty \right)$ to be of the form:

\begin{itemize}
\item Type I : $b_I(t) = c$, for some $c > 0$, with $I_{I} = \mathbb{R}$ and $d_{I}\coloneqq b_{I}^2 -nb_{I}'= c^2$.

\item Type II: $b_{II}(t) = n c \tan \left( c t \right)$, for some $c >0$, with $I_{II} = \left(0, \pi/2c \right)$ and $d_{II}\coloneqq b_{II}^2 -nb_{II}' = -c^2n^2$.
\item Type III: $b_{III}(t) = - n c \tanh \left( c t \right)$, for some $c >0$, with $I_{III} = \left( - \infty , 0  \right)$ and $d_{III}\coloneqq b_{III}^2 -nb_{III}' = c^2n^2$.

\end{itemize}

If $b$ is of Type I, we recover the case already studied in \cite{LOSemiriemannian}, so our main interest lies in the other families of warping functions. This discussion leads us quite naturally to the study of those GRW spacetimes whose warping function is of Type II or III.

One could also wonder what the minimal condition is for the general solution to the PDE \eqref{eq:pde2} not to depend explicitly on $t$. As a first step, in this paper we will restrict our attention to warping functions of Types II and III, as they let us find many interesting examples.

\begin{rem}
It is natural to ask whether the manifolds introduced in Section \ref{section:types} satisfy the Null Convergence Condition (NCC), a hypothesis included in the recent paper \cite{delira} by de Lira and Roing.
Recall that the Null Convergence Condition (NCC) requires \(\Ric(U,U) \ge 0\) for every lightlike vector \(U\).
In a GRW spacetime, a lightlike vector can be written as \(U = X + \partial_t\), where \(X\) is a unit spacelike vector tangent to \(N\). Then,
\[
\Ric(U,U) = \Ric^{g_N}(X,X) + (n-1)\,b^{-2}(b'^2 - b\,b'').
\]
This shows that whether the NCC holds depends both on the curvature of \(N\) and the choice of the warping function \(b\). For instance:
\begin{itemize}
\item If \(b(t) = nc\tan(ct)\), the NCC fails even if \(N\) is Ricci-flat.
\item If \(b(t) = -nc\tanh(ct)\), the NCC is satisfied when \(N\) is Ricci-flat.
\end{itemize}
\end{rem}

\section{Grim Reapers} \label{section:grimreapers}

As a first example of our new notion, we want to consider the analogues of the most well-known classical translators, namely the Grim Reapers. This corresponds to taking $N = \mathbb{R}^{n+1} = \mathbb{R}^n \times \mathbb{R}$ with $\mathbb{R}^n$ acting by translations on the first $n$ coordinates. The relevant Riemannian submersion $\mathbb{R}^{n+1} \to \mathbb{R}$ is, then, the projection to the last coordinate.

We can generalise this situation by letting $\left(N,g_N\right)$ be a product manifold $P \times \mathbb{R}$, with $(P,g_P)$ any Riemannian manifold. It is clear that the projection $\pi \colon N = P \times \mathbb{R} \to \mathbb{R}$ to the second factor is a Riemannian submersion, so $\alpha = 1$. The fibres of $\pi$ are totally geodesic submanifolds of $N$, so they have constant mean curvature $h = 0$ with respect to $\nabla \pi$. The ODE \eqref{eq:ode2} in this case reads
\begin{equation*}
    f'' = \left( 1 - \frac{f'^2}{\left( b \circ f \right)^2}\right) \left( b \circ f \right) \left[ \left( b \circ f_t\right)^2 - n \left( b' \circ f_t \right) \right]  + \frac{\left(b' \circ f \right)}{\left( b \circ f \right)} f'^ 2 \, .
\end{equation*}

As we discussed before, we are primarily interested in the cases where the solution to this ODE does not depend explicitly on $t$, to get a genuine solution to the mean curvature flow. This is why we will restrict to the cases $b = b_{II}$ and $b = b_{III}$, where the ODE reduces to
\begin{equation} \label{eq:ode_grim_reapers}
    f'' = \left( 1 - \frac{f'^2}{\left( b_{\star} \circ f \right)^2}\right) \left( b_{\star} \circ f \right) d_{\star}  + \frac{\left(b_{{\star}}' \circ f \right)}{\left( b_{\star} \circ f \right)} f'^ 2 \, ,
\end{equation}
where ${\star} \in \left\{ II, III \right\}$ -- \textit{vid.} definitions of $b_{\star}$ and $d_{\star}$ in Section \ref{section:types}.

\begin{rem}
    If we take $(P,g_P)$ to be $\mathbb{R}^{n-1}$ with the Euclidean metric, the $X$-solitons corresponding  to the solutions to the ODE \eqref{eq:ode_grim_reapers} will yield a natural generalisation of the classical Grim Reaper cylinders in $\mathbb{R}^{n+1}$ \cite{grimreapers, trivinogrimreapers}.
    \end{rem}

\begin{rem} \label{rem:causality}
Let $f$ be a solution to the ODE \eqref{eq:ode_grim_reapers} defined in a neighbourhood of $s \in \mathbb{R}$. Recall the definitions of $\nu$ \eqref{eq:definitionnu} and $W$ \eqref{eq:definitionW}  Then, for every $p \in P$, the graph of $f \circ \pi$ in the GRW spacetime $\mathcal{RW}\left( N , g_{N},I_{\star},b_{\star} \right)$ is
\begin{itemize}
\item timelike at $(p,s,f(s))$ iff $1 - f'(s)^2/\left( b_{\star}(f(s))\right)^2 <0$;
\item  lightlike at $(p,s,f(s))$ iff $1 - f'(s)^2/\left( b_{\star}(f(s)) \right)^2 =0$;
\item spacelike at $(p,s,f(s))$ iff $1 - f'(s)^2/\left( b_{\star}(f(s)) \right)^2 >0$.
\end{itemize}
\end{rem}

The reason why we are restricting our attention to \textit{graphical} $X$-solitons is motivated by the following example:

\begin{ex} \label{example:vertical}
For each $s_0 \in \mathbb{R}$, define $P_{s_0} = P \times \{s_0\}$. This is a totally geodesic submanifold of $N$, and hence $P_{s_0} \times I_{\star}$ is a totally geodesic submanifold of $\mathcal{RW}\left( N , g_{N},I_{\star},b_{\star} \right)$. So its mean curvature vector is $0$. Moreover, $\partial_t^{\perp}=0$ as well. Therefore, these \textit{vertical} submanifolds are examples of $X$-solitons which are not graphical. In Proposition \ref{prop:wing} we will show that these are the only non-graphical ones.
\end{ex}

Some observations about the ODE \eqref{eq:ode_grim_reapers} are in order.

\begin{lemma} \label{lemma:reversibility}
The ODE \eqref{eq:ode_grim_reapers} is reversible, \textit{i.e.}, if $f$ is a solution, then $g(s) \coloneqq f(-s)$ is a solution as well. 
\end{lemma}

\begin{lemma}\label{lemma:light}
Let $\star \in \{II,III\}$. For every $(s_0,y_0) \in \mathbb{R} \times I_{\star}$, there exist exactly two lightlike solutions $l^{\star}_{+}, l^{\star}_{-}$ such that $l^{\star}_{+} (s_0) = y_0 = l^{\star}_{-} (s_0)$. They are given by the explicit expressions:
\begin{align*}
    &l^{II}_{\pm} (s) = \frac{1}{c} \arcsin \left( \sin\left(c y_0\right) e^{\pm n c^2 \left(s-s_0\right)} \right)  , \\
    &l^{III}_{\pm} (s) = \frac{1}{c} \arcsinh \left( \sinh\left(c y_0\right) e^{\mp n c^2 \left(s-s_0\right)} \right)  .
\end{align*}
Moreover, if a solution $f$ to \eqref{eq:ode_grim_reapers} satisfies
\[
1 - \frac{f'(\widehat{s})^2}{\left( b_{\star}(f(\widehat{s}))\right)^2}= 0
\]
at some point $\widehat{s}$, then it is globally a solution to the ODE
\begin{equation}\label{eq:light_ode}
    1 - \frac{f'(s)^2}{\left( b_{\star}(f(s))\right)^2}= 0 \, .
\end{equation}
\end{lemma}
\begin{proof}
As in the final part of Section \ref{section:submersions}, one can find the lightlike solutions by solving the ODE \eqref{eq:light_ode}, obtaining the functions in the statement of the lemma. The final assertion follows from a straightforward calculation, showing that a solution to \eqref{eq:light_ode} is automatically a solution to \eqref{eq:ode_grim_reapers}, together with uniqueness of solutions to ODEs.
\end{proof}

\begin{lemma} \label{lemma:uniqueness}
    Let $\star \in \{II,III\}$ and $(s_0,y_0) \in \mathbb{R} \times I_{\star}$. Let $l^{\star}_{\pm}$ be the two lightlike solutions with $l^{\star}_{\pm}(s_0)=y_0$. Let $f\colon L \to I_{\star}$ be a solution to \eqref{eq:ode_grim_reapers} with $f(s_0)=y_0$ different from $l^{\star}_{\pm}$, where $s_0 \in L \subseteq \mathbb{R}$ is the maximal interval of definition of $f$. Then, the graph of $f$ only intersects the graphs of $l^{\star}_{\pm}$ at $(s_0,y_0)$.
    \end{lemma}
    \begin{proof}
    Recall Remark \ref{rem:causality}. Firstly, suppose that
    \begin{equation*}
     1 - \frac{f'(s_0)^2}{\left( b_{\star}(f(s_0))\right)^2}> 0 \, .
    \end{equation*}
This is equivalent to the fact that $\left(l^{\star}_{-}\right)'(s_0) < f'(s_0) < \left(l^{\star}_{+}\right)'(s_0)$.
Suppose that there exists $s_1 > s_0 \in L$ where $l^{\star}_{-}$ is defined such that $f(s_1) = l^{\star}_{-}(s_1)$. Now define
    \begin{equation*}
    s_{\times} \coloneqq \inf \left\{ s \in (s_0,s_1] \colon f(s) = l^{\star}_{-}(s)\right\}  .
    \end{equation*}
    Then, as $\left(l^{\star}_{-}\right)'(s_0) < f'(s_0)$, it is clear that $s_0 < s_{\times}$. Moreover,  $f(s_{\times}) = l^{\star}_{-}(s_{\times})$. We claim that
    \begin{equation*}
         1 - \frac{f'(s_{\times})^2}{\left( b_{\star}(f(s_{\times}))\right)^2} < 0 \, .
    \end{equation*}
By definition of $l_{\pm}^{\star}$, this is equivalent to $f'(s_{\times}) < \left(l^{\star}_{-}\right)'(s_{\times})$.
Note that, for every $s_0 < s < s_{\times}$, $\left(l^{\star}_{-}\right)(s) < f(s)$. Hence, $f'(s_{\times}) \leq (l^{\star}_{-})'(s_{\times})$. On the other hand, for $\delta > 0$ sufficiently small, $f(s_{\times} + \delta) < l^{\star}_{-}(s_{\times} + \delta)$, because otherwise we would have $f'(s_{\times}) \geq (l^{\star}_{-})'(s_{\times})$, which would force $f'(s_{\times}) = (l^{\star}_{-})'(s_{\times})$, which in turn would imply that this is true everywhere, by the final assertion of Lemma \ref{lemma:light}. So, the claim is proved.

Now, by Bolzano's theorem, there exists $s_0 < \widehat{s} < s_{\times}$ such that
    \begin{equation*}
        1 - \frac{f'(\widehat{s})^2}{\left( b_{\star}(f(\widehat{s}))\right)^2}= 0 \, .
    \end{equation*}
But this would again imply that $f$ is globally a timelike solution to the ODE \eqref{eq:light_ode} by Lemma \ref{lemma:light}, which contradicts the assumption at the beginning of the proof.

Roughly speaking, we have just proved that, if a solution starts spacelike, then it cannot cross the corresponding $l^{\star}_{-}$ in positive time. The other cases are analogous.
\end{proof}

\begin{lemma}\label{lemma:0_axis}
For all $\star \in \{II,III\}$, the solutions to \eqref{eq:ode_grim_reapers} with initial condition in $\mathbb{R} \times I_{\star}$ do not blow up to the $0$-axis.
\end{lemma}
\begin{proof}
The constant solution $f = 0$ is a solution to the ODE obtained by multiplying both sides of \eqref{eq:ode_grim_reapers} by $\left( b_{\star} \circ f \right)$, namely
\begin{equation}\label{star} (b_{\star}\circ f)f'' = \big( (b_{\star}\circ f)^2 -f'^2\big) d_{\star} +(b_{\star}'\circ f)f'^2 \, .
\end{equation}
\color{black}
Hence, if $f$ is a solution to \eqref{eq:ode_grim_reapers} with $\lim_{s \to s_0}f(s)=0$ for some $s_0 \in \mathbb{R}$, then $f$ extends to a solution $\tilde{f}$ to the ODE \eqref{star} with $\tilde{f}(s_0)=0$, and so we have that
\begin{equation*}
-\tilde{f}'(s_0)^2 d_{\star} + b_{\star}'(0) \tilde{f}'(s_0)^2 = 0 \, .
\end{equation*}
Therefore, either $\tilde{f}'(s_0) = 0$ or $d_{\star} = b'_{\star}(0)$. The latter never happens, as one can readily check for $\star = II,III$. If $\tilde{f}'(s_0)=0$, then $\tilde{f}=0$ by uniqueness, which is a contradiction.
\end{proof}

\begin{lemma} \label{lemma:minmax}
For $\star = II$ (resp. $\star = III$), if a solution to \eqref{eq:ode_grim_reapers} has a critical point, then it is an absolute maximum (resp. absolute minimum).
\end{lemma}
\begin{proof}
If, for some $s_0 \in \mathbb{R}$, a solution $f$ to \eqref{eq:ode_grim_reapers} satisfies $f'(s_0)=0$, then
\begin{equation*}
f''(s_0) = b_{\star} \left(f(s_0)\right) d_{\star} .
\end{equation*}
And $b_{\star} > 0$, $d_{II}<0$ and $d_{III}>0$. Hence, if $f$ has a critical point, it is a local maximum (resp. local minimum). This implies that $f$ has at most one critical point. Hence, this extreme has to be absolute.
\end{proof}

\begin{lemma} \label{lemma:limits}
Let $\star \in \{II,III\}$ and let $f$ be a solution to the ODE \eqref{eq:ode_grim_reapers} defined on $(a,+\infty)$ (resp $(-\infty,a)$), for some $a \in \mathbb{R}$. Then, the limit
\begin{equation*}
l \coloneqq \lim_{s \to \ + \infty} f (s) \hspace{.3cm} \left( \text{resp.} \lim_{s \to \ - \infty} f (s)\right)
\end{equation*}
exists and is one of the endpoints of the interval $I_{\star}$, \textit{i.e.}, if $\star = II$ then $l \in \{0, \pi/2c\}$ and if $\star=III$ then $l \in \{-\infty, 0\}$.
\end{lemma}
\begin{proof}
Note that this limit exists because, by Lemma \ref{lemma:minmax}, $f$ has at most one critical point and so, for $\abs{s}$ sufficiently large, $f'(s)$ has constant sign. Suppose for contradiction that
\begin{equation*}
    \lim_{s \to \ + \infty} f (s) \hspace{.3cm} \left( \text{resp.} \lim_{s \to \ - \infty} f (s)\right) \eqqcolon l \in I_{\star} \, .
\end{equation*}
Then, we would have that
\begin{equation*}
    \lim_{s \to \ + \infty} f' (s) = \lim_{s \to \ + \infty} f'' (s) \hspace{.3cm} \left( \text{resp.} \lim_{s \to \ - \infty} f' (s) = \lim_{s \to \ - \infty} f'' (s)\right) = 0 \, .
\end{equation*}
But now, taking these limits in the ODE, we obtain that $d_{\star} b_{\star} (l) = 0$, which is a contradiction.
\end{proof}

\begin{lemma} \label{lemmainverse}
The inverse function $\xi$ of an injective solution $f$ to the ODE \eqref{eq:ode_grim_reapers} satisfies the ODE
\begin{equation} \label{eq:ode_inverse}
\xi'' = \xi' \left[ \left( \frac{1}{b_{\star}^2} - \xi'^2 \right) d_{\star} b_{\star} - \frac{b_{\star}'}{b_{\star}} \right] .
\end{equation}
\end{lemma}
\begin{proof}
It is a direct application of the inverse function theorem.
\end{proof}

We can already obtain an important result:

\begin{prop}\label{prop:wing}
There are no winglike (\textit{vid.} \cite{Clutterbuck}, 2.2) solutions to the ODE \eqref{eq:ode_grim_reapers}.
\end{prop}
\begin{proof}
A winglike solution to \eqref{eq:ode_grim_reapers} corresponds to a solution $\xi$ of \eqref{eq:ode_inverse} with a critical point. By uniqueness of solution, this would imply that $\xi$ is identically constant, which is a contradiction.
\end{proof}

\begin{rem}
In fact, the constant solutions to \eqref{eq:ode_inverse} correspond to the solitons described in Example \ref{example:vertical}.
\end{rem}

Note that, in ODE \eqref{eq:ode_inverse}, we can define $\beta \coloneqq \xi'$, and $\beta$ satisfies the ODE
\begin{equation}\label{eq:ode_beta}
    \beta' = \beta \left[ \left( \frac{1}{b_{\star}^2} - \beta^2 \right) d_{\star} b_{\star} - \frac{b_{\star}'}{b_{\star}} \right] \, .
\end{equation}
It turns out that we can solve this ODE explicitly. The solutions are
\begin{align*}
&\beta^{II}_{\pm} = \frac{\pm 1}{\tan(c y) \sqrt{c^2 n^2 + c_{1}^{II} \sin^{2 n} (c y)}} \, , \\
&\beta^{III}_{\pm} = \frac{\pm 1}{\tanh(c y) \sqrt{c^2 n^2 + c_{1}^{III} \sinh^{2 n} (c y)}} \, ,
\end{align*}
for $c_{1}^{\star} \in \mathbb{R}$.

\begin{rem}\label{rem:constantsign}
The functions $\beta^{\star}_{\pm}$ have constant sign, and are never $0$, in their maximal intervals of definition. Therefore, if $\beta \colon L_1 \to \mathbb{R}$ is a solution to \eqref{eq:ode_beta} with initial condition $\beta(y_0) = \beta_0$, then
\begin{equation*}
    \xi(y) = \xi_0 + \int_{y_0}^{y} \beta(z) dz
\end{equation*}
is a solution to \eqref{eq:ode_inverse} with $\xi(y_0)=\xi_0$ and $\xi'(y_0)=\beta_0$ which is also defined in the whole of $L_1$, and is injective there. Its inverse function $f$, defined in the whole of $\xi(L_1)$, will satisfy the ODE \eqref{eq:ode_grim_reapers}, by the inverse function theorem.
\end{rem}

In order to study the solutions to our original ODE \eqref{eq:ode_grim_reapers}, it is enough to study the functions $\beta_{\pm}^{\star}$ in detail, as we now illustrate. See the qualitative behaviour of $\beta_{\pm}^{\star}$ in Figures \ref{fig:beta1}, \ref{fig:beta2}, \ref{fig:beta3}, \ref{fig:beta4}, \ref{fig:beta5} and \ref{fig:beta6}.

\begin{center}
    \begin{figure}[!htb]
        \begin{minipage}[b]{0.32\linewidth}
          \centering
          \includegraphics{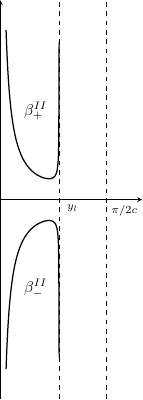}
          \captionsetup{width=0.9\linewidth}
          \caption[Short Caption]{$c_{1}^{II} < -c^2n^2$. }
          \label{fig:beta1}
        \end{minipage}\hfill
        \begin{minipage}[b]{0.32\linewidth}
          \centering
          \includegraphics{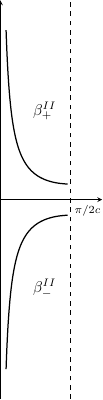}
          \captionsetup{width=0.9\linewidth}
          \caption{$c_{1}^{II} = -c^2n^2$. }
          \label{fig:beta2}
        \end{minipage} \hfill
        \begin{minipage}[b]{0.32\linewidth}
          \centering
          \includegraphics{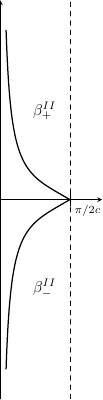}
          \captionsetup{width=0.9\linewidth}
          \caption{$c_{1}^{II} > -c^2n^2$. }
          \label{fig:beta3}
        \end{minipage}
\end{figure}
\end{center}

\begin{center}
    \begin{figure}[!htb]
    \begin{minipage}[h]{0.32\linewidth}
      \centering
      \includegraphics{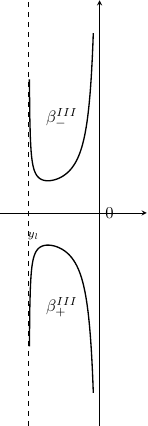}
      \captionsetup{width=0.8\linewidth}
      \caption{$c_{1}^{III} < 0$. }
      \label{fig:beta4}
    \end{minipage}
    \begin{minipage}[h]{0.32\linewidth}
      \centering
      \includegraphics{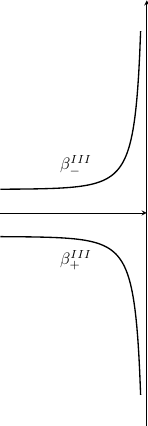}
      \captionsetup{width=0.8\linewidth}
      \caption{$c_{1}^{III} = 0$. }
      \label{fig:beta5}
    \end{minipage} 
    \begin{minipage}[h]{0.32\linewidth}
      \centering
      \includegraphics{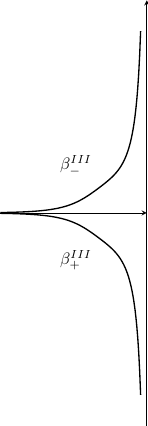}
      \captionsetup{width=0.8\linewidth}
      \caption{$c_{1}^{III} > 0$. }
      \label{fig:beta6}
    \end{minipage}
    \end{figure}
    \end{center}
\FloatBarrier

\begin{lemma} \label{lemma:critical point}
    Let $\star \in \{II,III\}$, and fix $(s_0,y_0) \in \mathbb{R} \times I_{\star}$.
    \begin{itemize}
    \item Let $f$ be the solution to the ODE \eqref{eq:ode_grim_reapers} with $f(s_0)=y_0$ and $f'(s_0) \neq 0$, with $\sign\left( f'(s_0) \right) = \pm 1$. Observe that $f$ is injective in a neighbourhood of $s_0$.
    \item Let $\beta$ be the solution to the ODE \eqref{eq:ode_beta} with $\beta(y_0) = 1/f'(s_0)$.
    \end{itemize}
    Then, $f$ has a critical point if, and only if, there exists $y_l$ in the maximal interval of definition of $\beta$ such that $\abs{\lim_{y \to y_l} \beta(y) } = \infty$.
\end{lemma}

\begin{proof}
It is a direct consequence of the inverse function theorem, together with Remark \ref{rem:constantsign}.
\end{proof}

Let us determine when the solutions to \eqref{eq:ode_grim_reapers} blow up in finite time. Let $\star \in \{II,III\}$, and fix $(s_0,y_0) \in \mathbb{R} \times I_{\star}$. Let $f$ be a solution to the ODE \eqref{eq:ode_grim_reapers} with $f(s_0)=y_0$. By Lemma \ref{lemma:minmax}, it has at most one critical point. Let $s_0 \in L \subseteq \mathbb{R}$ be its maximal interval of definition. There are two cases:

\begin{enumerate}

\item If $f$ has a critical point, then it is globally defined. Indeed, if $\star = II$, the critical point will be a maximum, so $f$ cannot blow up to $\pi/2c$, and it cannot blow up to $0$ by Lemma \ref{lemma:0_axis}. Similarly, if $\star=III$, the critical point will be a minimum, so $f$ cannot blow up to $-\infty$, and it cannot blow up to $0$ by Lemma \ref{lemma:0_axis}.

\item If $f$ does not have a critical point, then it is injective in $L$, $f'(s_0) \neq 0$, and the image of $f \colon L \to \mathbb{R}$ is the whole of $I_{\star}$, by Lemma \ref{lemma:limits}. Moreover, $f$ has an inverse $\xi$ which satisfies the ODE \eqref{eq:ode_inverse}, with initial conditions $\xi(y_0) = s_0$ and $\xi'(y_0) = 1/f'(s_0)$. So, if $\beta (y) \coloneqq \xi'(y)$, then $\beta$ satisfies the ODE \eqref{eq:ode_beta}, with initial condition $\beta(y_0) = 1/f'(s_0)$, by the inverse function theorem. Note that
\begin{align*}
    c_{1}^{II} &= - \frac{1-f'(s_0)^2/b_{II}(y_0)^2}{\sin^{2n}(c y_0)/c^2n^2} \, , \\
    c_{1}^{III} &= - \frac{1-f'(s_0)^2/b_{III}(y_0)^2}{\sinh^{2n}(c y_0)/c^2n^2} \, .
    \end{align*}
Therefore, if $\sign (f'(s_0)) = \pm 1$,
\begin{equation*}
\xi(y) = s_0 + \int_{y_0}^{y} \beta^{\star}_{\pm} (z) dz \, .
\end{equation*}
The maximal interval of definition of $f$ is the image of $\xi \colon I_{\star} \to \mathbb{R}$. This can be obtained by studying the convergence or divergence of the appropriate integrals, which we will do next.

\end{enumerate}

Recall that $f$ is timelike (resp. spacelike, lightlike) if, and only if, $c_{1}^{\star} > 0$ (resp. $c_{1}^{\star} < 0$, $c_{1}^{\star} = 0$).

Suppose that $c_{1}^{\star} > 0$, \textit{i.e.}, that $f$ is timelike. In this case, we see that $\beta^{\star}_{\pm}$ is defined in the whole of $I_{\star}$. The domain of $f$, as discussed before, is the image of the corresponding $\xi$, which is given by $\xi (y) = s_0 + \int_{y_0}^{y} \beta^{\star}_{\pm} (z) dz$. We now discuss the two cases $\star = II, III$ separately.

\begin{claim}
     For $\star=II$, $\int_{y_0}^{0} \beta^{II}_{\pm} (z) = \mp \infty$, while $\abs{\int_{y_0}^{\pi/2c} \beta^{II}_{\pm} (z) } < \infty $. Hence, if $f'(s_0) > 0$ (resp. $f'(s_0) < 0$), the maximal domain of $f$ is of the form $(-\infty, k)$ (resp. $(k, + \infty)$), with $\abs{k} < \infty $. This means that the timelike solutions blow up to $\pi/2c$ in finite time, and they do not blow up to $0$ in finite time.
\end{claim}

\begin{claimproof} \pushQED{\qed}
When $y \approx 0$, we have that
\begin{equation*}
\beta^{II}_{\pm} (y) \sim \frac{\pm 1}{\tan(cy)} \approx \frac{\pm 1}{cy} \, ,
\end{equation*}
and
\begin{equation*}
\int_{y_0}^{0} \frac{\pm 1}{cz} dz \hspace{.5cm}\text{   diverges.   }
\end{equation*}
However,
\begin{equation*}
    \lim_{y \to \left(\pi/2c\right)^-} \beta^{II}_{\pm} (z) dz = 0 \, ,
\end{equation*}
and so
\begin{equation*}
    \int_{y_0}^{\pi/2c} \beta^{II}_{\pm}(z) dz \hspace{.5cm}\text{   converges.   }
\end{equation*}
Note that this argument not only works when $c_{1}^{II} >0$, but, more generally, when $c_{1}^{II} > -c^2 n^2$. Recall Remark \ref{rem:constantsign} to deduce the claim on $f$. \qedhere
\popQED
\end{claimproof}

\begin{claim}
     For $\star=III$, $\abs{\int_{y_0}^{-\infty} \beta^{III}_{\pm} (z)} < \infty$, while $\int_{y_0}^{0} \beta^{III}_{\pm} (z) = \pm \infty $. Hence, if $f'(s_0) > 0$ (resp. $f'(s_0) < 0$), the maximal domain of $f$ is of the form $(k,\infty)$ (resp. $(-\infty,k)$), with $\abs{k} < \infty $. This means that the timelike solutions blow up to $-\infty$ in finite time.
\end{claim}

\begin{claimproof}
    When $y \to -\infty$, we have that
    \begin{equation*}
    \beta^{III}_{\pm} (y) \sim \frac{\pm 1}{\sinh^n(-cy)} \sim \pm e^{ncy} \, ,
    \end{equation*}
    and so
    \begin{equation*}
    \int_{y_0}^{-\infty} \beta^{III}_{\pm} (z) dz \hspace{.5cm}\text{   converges.   }
    \end{equation*}
    Moreover, when $y \to 0^-$,
    \begin{equation*}
        \pushQED{\qed}
        \beta^{III}_{\pm} (y) \sim  \frac{\pm 1}{\tanh(cy)} \sim \frac{\pm 1}{cy} \, ,
        \end{equation*}
        and
        \begin{equation*}
        \int_{y_0}^{0} \frac{\pm 1}{cz} dz \hspace{.5cm}\text{   diverges.   }
        \end{equation*}
    Recall Remark \ref{rem:constantsign} to deduce the claim on $f$. \qedhere
    \popQED

\end{claimproof}

Suppose now that $c_{1}^{\star} < 0$, \textit{i.e.}, that $f$ is spacelike.

\begin{itemize}
    \item For $\star=II$, there are three cases:
    \begin{enumerate}
        \item If $c_{1}^{II} < -c^2n^2$, then there is a point where $\beta^{II}_{\pm}$ blows up. By Remark \ref{rem:constantsign} and Lemma \ref{lemma:minmax}, $f$ has a maximum. \qed
        \item If $-c^2n^2 < c_{1}^{II} < 0$, then $\beta^{II}_{\pm}$ doesn't blow up in finite time. And, as in the case $c_{1}^{II}>0$, if $f'(s_0) > 0$ (resp. $f'(s_0) < 0$), we have that the domain of $f$ is of the form $(-\infty, k)$ (resp. $(k, \infty)$), with $\abs{k} < \infty $. This means that these solutions blow up to $0$ in finite time. \qed
        \item If $c_{1}^{II} = -c^2n^2$, then $\beta^{II}_{\pm}$ has a finite limit in $\pi/2c$. This gives us a monotonic solution that blows up to $\pi/2c$, and is the \textit{last} one to do so. 
        \begin{proof}
        An elementary calculation shows that
         \begin{equation*}
            \lim_{y \to \left(\pi/2c \right)^-} \beta_{\pm}^{II} (y) =  \lim_{y \to \left(\pi/2c \right)^-} \frac{\pm 1}{nc\tan(c y)\sqrt{1-\sin^{2n}(c y)}} = \frac{\pm 1}{cn\sqrt{n}}  \, .
         \end{equation*}
         Recall Remark \ref{rem:constantsign} to deduce the claim on $f$.
        \end{proof}
    \end{enumerate}

    \item For $\star=III$, we know that spacelike solutions are globally defined, as they cannot blow up in finite time because, by Lemma \ref{lemma:uniqueness}, they are bounded by the lightlike solutions, which are well-defined everywhere. But the corresponding $\beta$ blows up in finite time, which corresponds to $f$ having a critical point (Lemma \ref{lemma:critical point}). Therefore, by Remark \ref{rem:constantsign}, all spacelike solutions will have an absolute minimum (Lemma \ref{lemma:minmax}), and tend asymptotically to $0$ (Lemma \ref{lemma:limits}). \qed
\end{itemize}

We summarise the classification we have obtained in the following theorem, and illustrate it in figures \ref{fig:typeII} and \ref{fig:typeIII}. Note that the solutions $f$ to \eqref{eq:ode_grim_reapers} can be extended to maps $N\supseteq P\times L'\ni (p,s)\mapsto f(s)$, similarly to the classical Grim Reapers in $\mathbb{R}^3$.

\begin{thmm} \label{thm:maintheorem}
Let $n \geq 1$, $(P,g_P)$ an $n$-dimensional Riemannian manifold, and $\left( N = P \times \mathbb{R} , \, g_N = g_P + \diff s^2 \right)$. Given $\star \in \{II,III\}$, define $X_{\star} = b_{\star}(t) \, \partial_t$ a conformal Killing vector field on $\mathcal{RW} \left( N,g_N, I_{\star}, b_{\star} \right)$. Then, the $X_{\star}$-solitons which are graphs of functions of the form
\begin{align*}
P \times L' \subseteq N &\to I_{\star} \\
(p,s) &\mapsto f(s) \, ,
\end{align*}
where $L' \subseteq \mathbb{R}$ is an open interval, and $f$ has to be a solution to ODE \eqref{eq:ode_grim_reapers}. Denote the maximal interval of definition of such an $f$ by $L$. These $X_{\star}$-solitons are in one, and only one, of the following classes:
\begin{itemize}
    \item For $\star = II$,
\begin{enumerate}
    \item Type II.A: timelike, $L = (-\infty,k)$ or $(k,\infty)$ for some $k \in \mathbb{R}$, $f$ blows up to $\pi/2c$ at $k$ and tends asymptotically to $0$ on the other end of $L$.
    \item Type II.B: spacelike, $L = \mathbb{R}$, $f$ has a global maximum and tends asymptotically to $0$ in both ends of $L$.
    \item Type II.C: spacelike, $L = (-\infty,k)$ or $(k,\infty)$ for some $k \in \mathbb{R}$, $f$ blows up to $\pi/2c$ at $k$ and tends asymptotically to $0$ on the other end of $L$.

  \end{enumerate}
    \item For $\star=III$,
    \begin{enumerate}
        \item Type III.A: timelike, $L = (-\infty,k)$ or $(k,\infty)$ for some $k \in \mathbb{R}$, $f$ blows up to $-\infty$ at $k$ and tends asymptotically to $0$ on the other end of $L$.
        \item Type III.B: spacelike, $L = \mathbb{R}$, $f$ has a global minimum and tends asymptotically to $0$ in both ends of $L$. \qed
      \end{enumerate}
\end{itemize}
\end{thmm}

        \begin{figure}[h]
            \centering
            \includegraphics{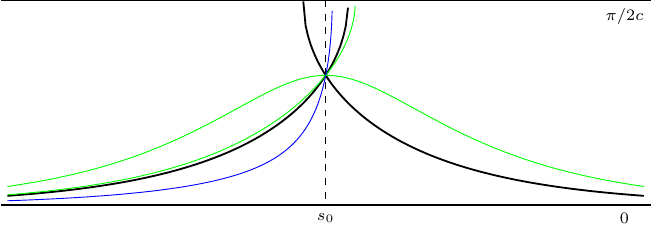}
            \caption{Some generating curves of Grim Reapers of Type II. In \textcolor{black}{black}, the lightlike solutions. In \textcolor{blue}{blue}, a timelike solution (Type II.A). In \textcolor{green}{green}, a spacelike solution with a maximum (Type II.B) and one which blows up to $\pi/2c$ (Type II.C). They have been obtained numerically using Wolfram Mathematica. }
            \label{fig:typeII}
            \end{figure}

        \begin{figure}[h]
            \centering
            \includegraphics{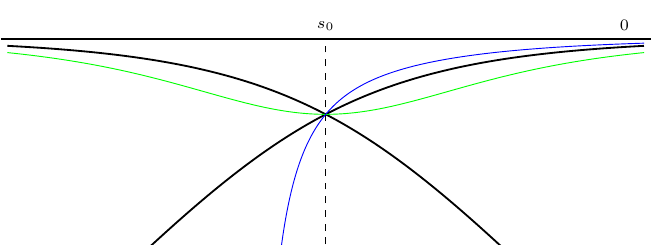}
            \caption{Some generating curves of Grim Reapers of Type III. In \textcolor{black}{black}, the lightlike solutions. In \textcolor{blue}{blue}, a timelike solution (Type III.A). In \textcolor{green}{green}, a spacelike solution (Type III.B). They have been obtained numerically using Wolfram Mathematica.  }
            \label{fig:typeIII}
            \end{figure}

\FloatBarrier

\section*{Acknowledgements}
M.-A. Lawn and M. Ortega are partially financed by the Spanish MICINN, project PID2020-116126GB-I00.

The authors would like to thank Miguel S\'anchez for his helpful comments on properties of Generalised Robertson-Walker spacetimes.

Finally, the authors would like to thank the referees for their valuable comments, which improved the paper.
\section*{Declarations}

\textbf{Data availability statement: }Data sharing not applicable to this article, as no datasets were generated or analysed during the current study.

\textbf{Conflict of interest statement: }On behalf of all authors, the corresponding author states that there is no conflict of interest.

\bibliographystyle{plainurl}
\bibliography{references.bib}
\end{document}